\documentclass{amsart}
\pdfoutput=1 
\usepackage[text={125mm,195mm}, centering]{geometry}
\usepackage{amssymb, latexsym, amsmath, verbatim, amsthm, amscd,xfrac}
\usepackage{pinlabel}
\usepackage{graphicx}
\usepackage{hyperref}

\title{Growth Tight Actions of Product Groups}
\author{Christopher H. Cashen}
\email{\href{mailto:christopher.cashen@univie.ac.at}{christopher.cashen@univie.ac.at}}
\address{Fakult\"at f\"ur Mathematik\\
Universit\"at Wien}
\author{Jing Tao} 
\email{\href{mailto:jing@math.ou.edu}{jing@math.ou.edu}}
\address{Department of Mathematics\\University of Oklahoma}
\date{February 4, 2015}

\thanks{The first author is supported by the European Research Council
(ERC) grant of Goulnara ARZHANTSEVA, grant agreement \#259527 and the Erwin
Schr\"odinger Institute workshop ``Geometry and Computation in Groups''.
The second author is partially supported by NSF grant DMS-1311834.}

\hypersetup{
    pdftitle={},    
    pdfauthor={Christopher H. Cashen and Jing Tao},     
    pdfkeywords={Growth tight, exponential growth rate, product groups}, 
    colorlinks=true,       
    linkcolor=black,          
    citecolor=black,        
    filecolor=black,      
    urlcolor=black           
}

\theoremstyle{plain}
\newtheorem{theorem}{Theorem}[section]
\newtheorem{lemma}{Lemma}[section]
\newtheorem{proposition}{Proposition}[section]
\newtheorem{corollary}{Corollary}[section]

\newtheorem*{question*}{Question}

\theoremstyle{remark}

\newtheorem*{remark}{Remark}

\theoremstyle{definition}
\newtheorem{definition}{Definition}[section]

\def\makeautorefname#1#2{\expandafter\def\csname#1autorefname\endcsname{#2}}
\let\fullref\autoref

\makeautorefname{theorem}{Theorem} 
\makeautorefname{lemma}{Lemma} 
\makeautorefname{proposition}{Proposition} 
\makeautorefname{corollary}{Corollary} 
\makeautorefname{definition}{Definition}
\makeautorefname{example}{Example}
\makeautorefname{section}{Section}
\makeautorefname{subsection}{Section}
\makeautorefname{subsubsection}{Section}
\makeautorefname{claim}{Claim}

\makeatletter 
\let\c@lemma=\c@theorem 
\makeatother
\makeatletter 
\let\c@proposition=\c@theorem 
\makeatother
\makeatletter 
\let\c@corollary=\c@theorem 
\makeatother
\makeatletter 
\let\c@definition=\c@theorem 
\makeatother
\makeatletter 
\let\c@example=\c@theorem 
\makeatother
\makeatletter 
\let\c@claim=\c@theorem 
\makeatother

\makeatletter
\newsavebox\myboxA
\newsavebox\myboxB
\newlength\mylenA

\newcommand*\xoverline[2][0.75]{%
    \sbox{\myboxA}{$\m@th#2$}%
    \setbox\myboxB\null
    \ht\myboxB=\ht\myboxA%
    \dp\myboxB=\dp\myboxA%
    \wd\myboxB=#1\wd\myboxA
    \sbox\myboxB{$\m@th\overline{\copy\myboxB}$}
    \setlength\mylenA{\the\wd\myboxA}
    \addtolength\mylenA{-\the\wd\myboxB}%
    \ifdim\wd\myboxB<\wd\myboxA%
       \rlap{\hskip 0.5\mylenA\usebox\myboxB}{\usebox\myboxA}%
    \else
        \hskip -0.5\mylenA\rlap{\usebox\myboxA}{\hskip 0.5\mylenA\usebox\myboxB}%
    \fi}
\makeatother
\newcommand{\inv}[1]{\xoverline{#1}}

\def\act{\curvearrowright} 
\def\from{\colon\thinspace}

\newcommand{\normal}{\trianglelefteq}
\renewcommand{\setminus}{-}

\def\bp{o}
\def\X{\mathcal{X}}
\def\P{\mathcal{P}}
\def\A{\mathcal{A}}
\def\H{\mathcal{H}}
\DeclareMathOperator{\diam}{diam}
\newcommand{\basis}{S}
\newcommand{\F}{\mathbb{F}}
\newcommand{\Y}{\mathbb{Y}}

\begin{document}
\begin{abstract}

A group action on a metric space is called growth tight if the
exponential growth rate of the group with respect to the induced
pseudo-metric is strictly greater than that of its quotients. 
A prototypical example is the action of a free group on its Cayley
graph with respect to a free generating set.
More generally, with Arzhantseva we have shown that group actions with strongly
contracting elements are growth tight. 

Examples of non-growth tight actions are product
groups acting on the $L^1$ products of Cayley graphs of the factors.

In this paper we consider actions of product groups on product spaces,
where each factor group acts with a strongly contracting element on
its respective factor space. 
We show that this action is growth tight with respect to the $L^p$
metric on the product space, for all $1<p\leq \infty$.
In particular, the $L^\infty$ metric on a product of Cayley graphs corresponds to a
word metric on the product group.
This gives the first examples of groups that are growth tight with
respect to an action on one of their Cayley graphs and non-growth
tight with respect to an action on another, answering a question of
Grigorchuk and de la Harpe.

\end{abstract}

\maketitle


\section{Introduction}
The \emph{growth exponent} of a set $A$ with respect to a pseudo-metric $d$
is
\[\delta_{A,d} =\limsup_{r\to\infty}\frac{1}{r}\cdot \log \# \{a\in A\mid
  d(\bp,a)\leq r\}\]
where $\#$ denotes cardinality and $\bp\in A$ is some basepoint. The limit
is independent of the choice of basepoint.

Let $G$ be a finitely generated group, and let $(\X,d, \bp)$ be a proper,
based, geodesic metric space on which $G$ acts properly discontinuously and
cocompactly by isometries. 

The metric $d$ induces a left invariant pseudo-metric $\bar{d}$ on any
quotient $G/N$ of $G$ by $\bar{d}(gN,g'N)=\min_{n,n'\in
N}d(gn.\bp,g'n'\!.\bp)$. When $(\X,d,\bp)$ is clear we let $\delta_{G/N}$
denote $\delta_{G/N, \bar{d}}$ and let $\delta_G$ denote $\delta_{G/\{1\},\bar{d}}$.

\begin{definition}[{\cite{ArzCasTao13}}]
  $G\curvearrowright \X$ is a \emph{growth tight action} if
  $\delta_G>\delta_{G/N}$ for every infinite normal subgroup $N\normal G$.

  If $\basis$ is a finite generating set of $G$, we say $G$ is \emph{growth
  tight with respect to $\basis$} if the action of $G$ via left
multiplication on the Cayley graph of $G$
  with respect to $\basis$ is growth tight.
\end{definition}

The first examples of such actions were given by Grigorchuk and de la Harpe
\cite{GriDeL97}, who showed that a finite rank, non-abelian free group $\F$
is growth tight with respect to any free generating set $\basis$. In the
same paper, they observe that the product $\F\times\F$ is not growth tight
with respect to the generating set
$\basis\times\{1\}\cup\{1\}\times\basis$, and ask whether there exists a
finite generating set with respect to which $\F\times\F$ is growth tight.

We answer this question affirmatively. This is the first example of a group
that is growth tight with respect to one generating set and not growth
tight with respect to another.

Our main result is for growth tightness of product groups $G_1 \times
\cdots \times G_n$. We require that each factor $G_i$ acts cocompactly with
a strongly contracting element on a space $\X_i$, see
\fullref{def:strongcontracting}. Examples include actions of hyperbolic or
relatively hyperbolic groups by left multiplication on any of their Cayley
graphs, and groups acting cocompactly on proper CAT(0) spaces with rank 1
isometries. With Arzhantseva \cite{ArzCasTao13}, we have shown that such
actions are growth tight. 

\begin{theorem}\label{main} 
  For $1\leq i\leq n$, let $G_i$ be a non-elementary, finitely generated
  group acting properly discontinuously and cocompactly by isometries on a
  proper, based, geodesic metric space $(\X_i,d_i,\bp_i)$ with a strongly
  contracting element. Let $G=G_1\times\cdots\times G_n$. Let
  $\X=\X_1\times\cdots\times \X_n$, with $\bp=(\bp_1,\dots, \bp_n)$ and let
  $d$ be the $L^p$ metric on $\X$ for some $1 \le p \le \infty$. 
  Let $G\act \X$ be the coordinate-wise action.
  Then $G\act \X$ is growth tight unless $p=1$ and $n>1$.
\end{theorem}

\begin{remark}
  Cocompactness of the factor actions is not strictly necessary. We use it
  to prove a subadditivity result, \fullref{lem:subadditivity}. There are
  weaker conditions than cocompactness of the action that can be used to
  prove such a result. These are discussed in \cite[Section 6]{ArzCasTao13}. For
  simplicity, we will stick to cocompact actions in this paper, since this
  suffices for our main applications.
\end{remark}

In the case that $\X_i$ is the Cayley graph of $G_i$ with respect to a
finite, symmetric generating set $\basis_i$, there is a natural bijection
between vertices of $\X$ and elements of $G$. This bijection is an isometry
between vertices of $\X$ with the $L^1$ metric and elements of $G$ with the
word metric corresponding to the generating set: \[\basis^1=\bigcup_{1\leq
i\leq n} \{(s_1,\dots,s_n)\mid s_j=1 \text{ for }j\neq i\text{ and } s_i\in
S_i\}\] The same bijection is also an isometry between vertices of $\X$
with the $L^\infty$ metric and elements of $G$ with the word metric
corresponding to the generating set: \[ \basis^\infty = \big\{
(s_1,\ldots,s_n) \mid s_i \in \basis_i
\cup \{1\} \big \}\]

\begin{corollary}
For $1\leq i\leq n$, let $G_i$ be a non-elementary group with a finite,
symmetric generating set $S_i$. Let $\X_i$ be the Cayley graph of $G_i$
with respect to $S_i$, and suppose that the action of $G_i$ on $\X_i$ by
left multiplication has a strongly contracting element. When $n\geq 2$, the product
$G=G_1\times\cdots \times G_n$ admits a finite generating set $S^1$ for
which the action on the corresponding Cayley graph is not growth tight and
another finite generating set $S^\infty$ for which the action on the
corresponding Cayley graph is growth tight. 
\end{corollary}

Non-elementary, finitely generated, relatively hyperbolic groups, and
finite rank free groups in particular, act with a strongly
contracting element on any one of their Cayley graphs, so:

\begin{corollary}\label{corollary:fxf}
  If $\F$ is a finite rank free group and $\basis$ is a finite, symmetric
  free generating  set of $\F$ then $\F\times\F$ is growth tight with respect to
  the generating set $(\basis\cup\{1\})\times (\basis\cup\{1\})$.
\end{corollary}

Another common way to think of $\F\times\F$ is as the Right Angled
Artin Group with defining graph the join of two sets of vertices of
cardinality equal to the rank of $\F$. 
The universal cover of the corresponding Salvetti complex is the
product of Cayley graphs of $\F$ with respect to free generating
sets. 
There are two natural metrics to consider on the vertex set of the
universal cover of the Salvetti complex: the induced length metric
from the piecewise Euclidean structure, which is the restriction of
the $L^2$ metric on the product, and the induced length metric in the
1--skeleton, which is the restriction of the $L^1$ metric on the
product. 

\begin{corollary}
  The action of $\F\times\F$ on the universal cover of its Salvetti
  complex is growth tight with respect to the piecewise Euclidean
  metric but not growth tight with respect to the 1--skeleton metric.
\end{corollary}

We sketch a direct proof of \fullref{corollary:fxf}. The proof of \fullref{main}
 follows the same outline. 

\begin{proof}[Sketch proof of \fullref{corollary:fxf}] 
  
  Let $\X$ be the Cayley graph of $\F$ with respect to $\basis$. Let
  $G=\F\times\F$ be generated by $(\basis\cup\{1\})\times
  (\basis\cup\{1\})$, which induces the $L^\infty$ metric on $\X \times
  \X$. We have $\delta_{G}=2\delta_\F>0$.

  Let $N$ be a non-trivial normal subgroup of $G$. If $N$ has trivial
  projection to, say, the first factor, then $G/N=\F\times(\F/\pi_2(N))$.
  Since $\F$ is growth tight with respect to every word metric,
  $\delta_{\F/\pi_2(N)}<\delta_\F$, so
  $\delta_{G/N}=\delta_\F+\delta_{\F/\pi_2(N)}<2\delta_\F=\delta_G$.

  If $N$ has non-trivial projection to both factors, then there is an element
  $(h_1,h_2)\in N$ with both coordinates non-trivial. For each $(a_1,a_2)N\in
  (\F\times\F)/N=G/N$, choose an element $(a_1',a_2')\in(a_1,a_2)N$ such that
  \[ d \big( (a_1',a_2'),(1,1) \big)= d \big( (a_1,a_2)N,(1,1) \big). \]
  Let $A=\{(a_1',a_2')\mid (a_1,a_2)N\in G/N\}$. We call $A$ a
  \emph{minimal section of the quotient map}. We have
  $\delta_{A,d}=\delta_{G/N, \bar{d}}$.

  Given a non-trivial, reduced word $f$, let $W(f)$ be the subset of elements
  of $\F$ whose expression as a reduced word in $\basis$ contains $f$ as a
  subword. Denote by $\inv{a}$ the inverse of a word $a$ in $\F$. If
  $(a_1',a_2')\in W(h_1)\times W(h_2)$ then there exist $b_i$ and $c_i$
  such that $a_i'=b_ih_ic_i$ for $i=1,\,2$, and
  \[ 
    (a_1',a_2')=(b_1h_1c_1,b_2h_2c_2) =
    (b_1c_1,b_2c_2)\cdot(\inv{c_1}h_1c_1,\inv{c_2}h_2c_2)
  \] 
  So $(b_1c_1,b_2c_2)N= (a_1',a_2')N$, but this contradicts the fact that
  $(a_1',a_2')\in A$, since $|(b_1c_1,b_2c_2)|_\infty<|(a_1',a_2')|_\infty$.
  Therefore, $A\subset (\F\setminus W(h_1))\times\F\cup\F\times (\F\setminus
  W(h_2))$. However, for any non-trivial $f$ the growth exponent of
  $\F\setminus W(f)$ is strictly less than that of $\F$, so the growth
  exponent of $A$ is strictly less than that of $\F\times\F$. \qedhere

\end{proof}

The fact that the growth exponent of $F\setminus W(f)$ is strictly less
than that of $F$ has analogues in formal language theory. A language
$\mathcal{L}$ over a finite alphabet is known as `growth-sensitive' or
`entropy-sensitive' if for every finite set of words in $\mathcal{L}$,
called the \emph{forbidden words}, the sub-language of words that do not
contain one of the forbidden words as a subword has strictly smaller growth
exponent than $\mathcal{L}$. It has been a topic of recent interest to
decide what kinds of languages are growth-sensitive
\cite{CecWoe02,CecWoe03,HusSavWoe10}.

Our approach to growth tightness is to prove a coarse-geometric version of
growth sensitivity, where the forbidden word is a power of a strongly
contracting element.

The first coarse-geometric version of growth sensitivity was used by
Arzhantseva and Lysenok \cite{ArzLys02} to prove growth tightness for
hyperbolic groups.
With Arzhantseva, \cite{ArzCasTao13} we gave a more general
construction that applied to group actions with strongly contracting
elements. 
The idea is that the action of a strongly contracting element closely
resembles the action of an infinite order element of a hyperbolic
group on a Cayley graph.

In \cite{ArzCasTao13} we proved a coarse-geometric version of the statement
that the growth exponent of the set of reduced words in $\F$ that do not
contain $f$ or $\inv{f}$ as subwords is strictly less than the growth
exponent of $\F$. For products this is not enough, since, for example, if
$(f,f)\in N\normal\F\times\F$ we cannot make the element $(f,\inv{f})$
shorter by applying powers of $(f,f)$. We really want to forbid only
positive occurrences of $f$ in each coordinate, so we need to strengthen
our coarse-geometric statement to take orientation into account.

After preliminaries in \fullref{sec:prelim}, we show in
\fullref{sec:contractingineachcoordinate} that an infinite normal subgroup
of $G$ that has infinite projection to each factor contains an element
$\underline{h}$ for which each coordinate is strongly contracting for the
action of the factor group on the factor space.

In \fullref{sec:nolongprojection} we prove the main technical lemma,
\fullref{mainlemma}, which is our oriented growth sensitivity result.

In \fullref{sec:proof} we complete the proof of \fullref{main}.

\section{Preliminaries}\label{sec:prelim}

For any group $G$, we use $\inv{g}$ to denote the multiplicative inverse of
$g\in G$.

A group is \emph{elementary} if it is finite or has an infinite cyclic
subgroup of finite index.

A \emph{quasi-map} $\pi\from\X\to\mathcal{Y}$ between metric spaces 
assigns to each point $x\in\X$ a subset $\pi(x)\subset \mathcal{Y}$
of uniformly bounded diameter.

\subsection{Strongly Contracting Elements}

We define strongly contracting elements following Sisto\footnote{Sisto
  considers `$\mathcal{PS}$--contracting projections'. 
We use `strong' to indicate the special case that
  $\mathcal{PS}$ is the collection of all geodesic segments in $\X$.}
\cite{Sis11}. See also \cite{ArzCasTao13} for additional reference.

\begin{definition}
Let $(\X,d)$ be a proper geodesic metric space, and
  let $\A \subset \X$ be a subset.
Given a constant $C > 0$, a map $\pi_\A \colon \X \to \A$ is called a
\emph{$C$--strongly contracting projection} if $\pi_\A$ satisfies the
following properties:
\begin{itemize}
  \item For every $a \in \A$, $d \big( a, \pi_\A(a) \big) \le C$.
  \item For every $x, y \in \X$, if $d \big( \pi_\A(x), \pi_\A(y) \big) >
  C$, then for every geodesic segment $\P$ with endpoints $x$ and $y$, we have
  $d(\pi_\A(x), \P) \le C$ and $d(\pi_\A(y), \P) \le C$. 
\end{itemize}

We say the map $\pi_\A$ is a \emph{strongly contracting projection} if
it is $C$--strongly contracting for some $C>0$.
\end{definition}

Fix a base point $\bp \in \X$. Let $G$ be a finitely-generated group that
admits a proper, cocompact, and isometric action on $\X$. 

\begin{definition} \label{def:strongcontracting}
An element $h \in G$ is $C$--\emph{strongly contracting} if $i\mapsto h^i\!.\bp$
is a quasi-geodesic and if there exists $C > 0$ such that, for every
geodesic segment $\P$ with endpoints on $\langle h \rangle.\bp$, there exists a
$C$--strongly contracting projection $\pi_\P \from \X \to \P$. 

An element $h \in G$ \emph{strongly contracting} if there exists a
$C>0$ such that $h$ is $C$--strongly contracting.
\end{definition}

The property of strongly contracting is independent of the base point
$\bp$. 
Since the action is by isometries, a conjugate of a strongly
contracting element is strongly contracting.

Let $h \in G$ be a strongly contracting element. Let $E(h) < G$ be the subgroup such
that $g \in E(h)$ if and only if the Hausdorff distance between $\langle h
\rangle. o$ and $g \langle h \rangle .o$ is bounded. Then $E(h)$ is
\emph{hyperbolically embedded} in the sense of Dahmani-Guirardel-Osin
\cite{DahGuiOsi11}, and $E(h)$ is
the unique maximal virtually cyclic subgroup containing $h$ \cite[Lemma
6.5]{DahGuiOsi11}.
Thus, $E(h)$ is the subgroup that often called the
\emph{elementarizer} or \emph{elementary closure} of $\langle h\rangle$.

\begin{definition} \label{def:quasiaxis}
  
  Given a strongly contracting element $h \in G$ and a point $\bp \in \X$, the set
  $\H = E(h).o$ is called a \emph{quasi-axis} in $\X$ for $h$.

\end{definition}

\begin{lemma}[{\cite[Lemma~2.20]{ArzCasTao13}}] \label{lem:contraction}
  
  If $h \in G$ is strongly contracting, then there exists a
  strongly contracting projection quasi-map $\pi_{\H} \from G \to \H$
  such that $\pi_{ \H}$ is $E(h)$--equivariant.

\end{lemma}

\begin{definition}\label{def:equiv}
 If $h\in G$ is strongly contracting and $g\notin E(h)$ define
 $\pi_{g\H}\from \X\to g\H$ by $\pi_{g\H}(x)=g.\pi_\H(\inv{g}.x)$.
\end{definition}

Combining \fullref{lem:contraction} and \fullref{def:equiv}, we may
assume that the strongly contracting projection quasi-maps $\pi_{g\H}$
to translates of $\H$ are $G$--equivariant.

\begin{lemma}\label{lem:almostgeodesic}
  If $h\in G$ is $C$--strongly contracting there exist
  non-negative constants $\lambda$, $\epsilon$, and $\mu$ such that $i\to
  h^i\!.\bp$ is a $(\lambda,\epsilon)$--quasi-geodesic and for
  $0\leq\alpha\leq\beta$ every geodesic from $\bp$ to $h^\beta\!.\bp$ passes
  within distance $\mu$ of $h^\alpha\!.\bp$.
\end{lemma}
\begin{proof}
There exist $\lambda$ and $\epsilon$ such that $i\to
  h^i\!.\bp$ is a $(\lambda,\epsilon)$--quasi-geodesic by definition
  of contracting element.
Let $\gamma$ be a geodesic segment from $\bp$ to $h^{\beta}\!.\bp$.
Then $\gamma$ is Morse, by \cite[Lemma~2.9]{Sis11}.
Thus, there is a $\mu$ depending only on $C$, $\lambda$, and
$\epsilon$ such that every $(\lambda,\epsilon)$--quasi-geodesic
segment with endpoints on $\gamma$ is contained in the
$\mu$--neighborhood of $\gamma$.
But $i\mapsto h^i\!.\bp$ for $i\in [0,\beta]$ is such a
$(\lambda,\epsilon)$--quasi-geodesic, so there is a point of $\gamma$
at distance at most $\mu$ from $h^\alpha\!.\bp$.
\end{proof}

\subsection{Actions on Quasi-trees}
Let $h$ be a contracting element for $G\act\X$ as in the previous
section, and let $\H$ be the quasi-axis of $h$.

In \fullref{mainlemma} we will consider a \emph{free product subset}
\[ Z^*\!*h^m = \bigcup_{i=1}^\infty \big\{ z_1h^m\cdots z_ih^m \mid z_j\in
  Z\setminus \{1\} \big\}\]
for a certain subset $Z\subset G$ and a sufficiently large $m$.
We wish to know that the orbit map from $G$ into $\X$ is an embedding
on this free product set.

This statement recalls the following well known result:
\begin{proposition}[Baumslag's Lemma \cite{Bau62}]
  If $z_1,\dots, z_k$ and $h$ are elements of a free group such that $h$
  does not commute with any of the $z_i$, then $z_1 h^{m_1}\cdots
  z_kh^{m_k}\neq 1$ if all the $|m_i|$ are sufficiently large.
\end{proposition}

A convenient way to prove such an embedding result is to work in a
tree, so that the global result, that $z_1 h^{m_1}\cdots
  z_kh^{m_k}\neq 1$, can be certified by a local `no-backtracking'
  condition. 
In our situation, we do not have an action on a tree to work with, but
a construction of Bestvina, Bromberg, and Fujiwara \cite{BesBroFuj10}
produces an action of $G$ on a quasi-tree, a space quasi-isometric to
a simplicial tree, from the action of $G$ on
the $G$--translates of $\H$.
In \cite{ArzCasTao13} we use this quasi-tree construction and a
no-backtracking argument to prove that the orbit map is an embedding of a certain free
product subset. 
The proof of \fullref{mainlemma} consists of choosing an appropriate
free product set to which we can apply the argument from
\cite{ArzCasTao13}.
The details of the construction of the quasi-tree and the proof of the
free product subset embedding are somewhat technical, so we will not
repeat them here (see \cite[Section 3]{BesBroFuj10} and \cite[Section~2.4]{ArzCasTao13} for more details). However, we will make
use of some of Bestvina, Bromberg, and Fujiwara's `projection axioms',
which hold for quasi-axes of contracting elements by work of Sisto
\cite{Sis11}, as recounted below.

Let $\Y$ be the collection of all distinct $G$--translates of $\H$. 
For each $Y \in \Y$, let $\pi_Y$ be the projection map from the above. Set
\[ d^\pi_Y(x,y) = \diam  \{\pi_Y(x), \pi_Y(y)\} . \]

\begin{lemma}[{\cite[Section~2.4]{ArzCasTao13}, cf \cite[Theorem 5.6]{Sis11}}] \label{lem:projection}
There exists $\xi\geqslant 0$ such that for all distinct $X,Y,Z\in\mathbb{Y}$:
  \begin{enumerate}
\item[(P0)]  $\diam\pi_Y(X)\leqslant \xi$
\item[(P1)] At most one of $d_X^\pi(Y,Z)$, $d_Y^\pi(X,Z)$ and
  $d_Z^\pi(X,Y)$ is strictly greater than $\xi$.
\item[(P2)] $|\{V\in\mathbb{Y}\mid
  d^\pi_V(X,Y)>\xi\}|<\infty$
  \end{enumerate}
\end{lemma}

\section{Elements that are Strongly Contracting in each Coordinate}\label{sec:contractingineachcoordinate}

Let $G$ be a finitely generated, non-elementary
group acting properly discontinuously and cocompactly by isometries on a
based proper geodesic metric space $(\X,d,\bp)$ such that there exists an
element $h\in G$ that is strongly contracting for $G\act\X$. Let
$\H=E(h).\bp$. Let $C$ be the contraction constant for $\pi_\H$
from \fullref{lem:contraction}, and let $\xi$ be the constant of
\fullref{lem:projection}. For any $x \in \X$ and any $r > 0)$, denote by
$B_r(x)$ the open ball of radius $r$ about $x$.

\begin{lemma}\label{lem:powers}
  Let $p$ be a point of $\H$. Let $g$ be an element of $G$. There exists a
  constant $D$ such that either some non-trivial power of
  $g$ is contained in $\left<h\right>$ or for all $n>0$ we have
  $d_\H^\pi(g^n\!.p,p)\leq 2d(p,g.p)+D$.
\end{lemma}

\begin{proof}
Since $\left<h\right>$ is a finite index subgroup of $E(h)$, if some non-trivial
  power of $g$ is contained in $E(h)$ then some non-trivial power of
  $g$ is contained in $\langle h\rangle$, and we are done. 
Thus, we may assume that no non-trivial power of $g$
  is contained in $E(h)$. 
This implies that if $m\neq n$ then $g^m\H\neq
  g^n\H$. 

  Let $z$ be a point on a geodesic from $p$ to $g.p$ in
  $B_C(\pi_\H(g.p))$.
Let $\xi$ be the constant of \fullref{lem:projection}.
Axiom (P0) of \fullref{lem:projection} says $\diam \pi_\H(g\H)\leq\xi$.
  \begin{align*}
    d(p,g.p)
    & = d(p,z)+d(z,g.p) \\
    & \geq d \big( p,\pi_\H(g.p) \big) - C +d(z,g.p)\\
    & \geq d^\pi_\H(p,g\H) - C - \xi +d(z,g.p)\\
&\geq d^\pi_\H(p,g\H) - C - \xi
  \end{align*}

  By a similar argument, for every $k\neq 0,\pm 1$,
  \[ d(p,g.p)\geq d^\pi_{g^k\H}(\H,g^{\pm 1}\H)-2C-2\xi\] 
  
  Using the above we obtain that, for any $n>1$, 
  \begin{align*}
    d^\pi_{\bar{g}\H}(\H,g^{n-1}\H)
    &=d^\pi_{\H}(g\H,g^n\H)\geq d^\pi_{\H}(g^n\H,p)-d^\pi_{\H}(g\H,p)\\
    &\geq d^\pi_{\H}(g^n\H,p)-d(g.p,p)-C-\xi
  \end{align*}
  Suppose that $d^\pi_{\H}(g^n\H,p)-d(g.p,p)-C-\xi>\xi$. 
The previous inequality says $d^\pi_{\bar{g}\H}(\H,g^{n-1}\H)>\xi$, so
 (P1)
  of \fullref{lem:projection} implies $\xi \geq
  d^\pi_{\H}(\inv{g}\H,g^{n-1}\H)$, hence:
  \begin{align*}
    \xi \geq d^\pi_{\H}(\inv{g}\H,g^{n-1}\H)
    &\geq
    d^\pi_{\H}(g^n\H,p)-d^\pi_{\H}(p,\inv{g}\H)-d^\pi_{\H}(g^n\H,g^{n-1}\H)\\
    &\geq
    d^\pi_{\H}(g^n\H,p)-d^\pi_{\H}(p,\inv{g}\H)-d^\pi_{\bar{g}^n\H}(\H,\inv{g}\H)\\
    &\geq d^\pi_{\H}(g^n\H,p)-2d(g.p,p)-3C-3\xi
  \end{align*}

  Thus, $d^\pi_{\H}(g^n\H,p)\leq 2d(g.p,p)+D$ for $D=3C+4\xi$.
\end{proof}

\begin{lemma}\label{lem:uniformpowers}
  For every $g\in G$ there exists an $l> 0$ and an $n'\geq 0$ such that for
  all $m>0$ and all $n\geq n'$, except possibly one, the elements
  $g^{lm}h^n$ and $h^ng^{lm}$ are strongly contracting.
\end{lemma}

\begin{proof}
  Suppose there exists a minimal $a>0$ and $b$ such that $g^a=h^b$. If
  $b>0$ let $l=a$ and let $n'=0$, so that $g^{lm}h^n=h^{bm+n}$ is a
  positive power of $h$. If $b=0$ let $l=a$ and $n'=1$ so that
  $g^{lm}h^n=h^n$ is a positive power of $h$.
  If $b< 0$ let $l=a$, $n'=0$, and $n\geq n'$ such that $n\neq
  -mb$. Then $g^{lm}h^n$ is a non-zero power of $h$.  

  If no non-trivial power of $g$ is contained in $\left<h\right>$, let $l=1$. By
  \fullref{lem:powers}, there exists a $D'$ such that for every $p\in \H$
  and every $m>0$ we have $d_\H^\pi(g^m.p,p)\leq 2d(p,g.p)+D'$. Let $D$ be
  the maximum of $D'$ and the constant $D$ from \cite[Lemma~5.2]{Sis11}.
  Let $p\in \H$ be a point such that $d(p,g.p)$ is minimal. Let $n'$ be
  large enough so that $d(p,h^{n'}\!\!.p)\geq 4d(p,g.p)+3D$. Then for $n\geq
  n'$ we have $d(p,h^n\!.p)\geq
  d(\pi_\H(g^m\!.p),p)+d(p,\pi_\H(\inv{g}^{m}\!.p))+D$. This implies
  $g^{lm}h^n$ is strongly contracting by \cite[Lemma~5.2]{Sis11}.
  $h^ng^{lm}$ is also strongly contracting as it is conjugate to
  $g^{lm}h^n$.
\end{proof}

For $i=1,\dots, n$, let $G_i$ be a non-elementary group acting properly
discontinuously and cocompactly by isometries on a proper, based, geodesic
metric space $(\X_i,d_i,\bp_i)$.
Assume, for each $i$, that $G_i\act \X_i$ has a strongly contracting element. 
Let $G=G_1\times\cdots\times G_n$.
Let $\chi_i\from G\to G_i$ be
projection to the $i$--th coordinate.

\begin{lemma}\label{lemma:contracting}
  Let $N$ be an infinite normal subgroup of $G$ such that $\chi_i(N)$ is
  infinite for all $i$. There exists an element $\underline{h}=(h_1,\dots,
  h_n)\in N$ such that $h_i$ is a strongly contracting element for $G_i
  \act X_i$.
\end{lemma}

\begin{proof}
  $\chi_i(N)$ is an infinite normal subgroup of $G_i$, so it contains a
  strongly contracting element by
  \cite[Proposition~3.1]{ArzCasTao13}. 
For each $i$, let
  $\underline{g}_i=(a_{i,1},\dots, a_{i,n})\in N$ such that $a_{i,i}$ is a
  strongly contracting element for $G_i\act\X_i$.

  We will show by induction that there is a product of the
  $\underline{g}_i$ that gives the desired element $\underline{h}$. 
The element $\underline{g}_1$ has a strongly contracting element in
its first coordinate.
Suppose
  that there is a product $\underline{f}=(f_1,\dots,f_n)$ of
  $\underline{g}_1,\dots,\underline{g}_i$ such that the first $i$
  coordinates are strongly contracting elements in their coordinate spaces.

  For $1\leq j \leq i$ there exists an $l_j$ and  an $n'_j$ as in
  \fullref{lem:uniformpowers} such that for all $m$ and all $n\geq n'_j$,
  except possibly one, we have $a_{i+1,j}^{l_jm}f_j^{n}$ is strongly
  contracting. Similarly, there are $l_{i+1}$ and $n'_{i+1}$ such that
  $a_{i+1,i+1}^nf_{i+1}^{l_{i+1}m}$ is strongly contracting for all $m>0$
  and $n\geq n'_{i+1}$.

  Let $l$ be the least common multiple of $l_1,\dots, l_i$. Let $m$ be
  large enough so that $ml\geq n_{i+1}'$. Let
  $\lambda_k=l_{i+1}(k+\max_{j=1,\dots,i} n'_j)$, where $k\geq 0$ varies.

  Consider $\underline{g}_{i+1}^{ml}\underline{f}^{\lambda_k}$. For $1\leq
  j\leq i$, the $j$-th coordinate is strongly contracting for all except
  possibly one value of $k$, since $ml$ is a multiple of $l_j$ and
  $\lambda_k\geq n'_j$. Similarly, the $(i+1)$--st coordinate is strongly
  contracting for all except possibly one value of $k$ since $\lambda_k$ is
  a multiple of $l_{i+1}$ and $ml\geq n_{i+1}'$. By choosing a $k$ that is
  not among the at most $i+1$ forbidden values, we have that the first
  $i+1$ coordinates of $\underline{g}_{i+1}^{ml}\underline{f}^{\lambda_k}$
  are strongly contracting in their coordinate space.
\end{proof}

We will say an element $g \in G_i$ has a $K$--long $h_i$--projection if
there exists an $f\in G_i$ such that
$d^\pi_{f\H_i}(\bp_i,g.\bp_i) \geq K$.

\begin{lemma}\label{lem:independentcontractingelements}
  Given $\underline{h}$ as in \fullref{lemma:contracting}, there exists an
  element $\underline{h'}=(h'_1,\dots,h'_n)\in N$ such that $h'_i$ is
  strongly contracting for each $G_i \act X_i$ and there exists a $K$ such
  that powers of $h'_i$ have no $K$--long $h_i$--projections and powers of
  $h_i$ have no $K$--long $h'_i$--projections.
\end{lemma}

\begin{proof}
  For each $i$, the group $G_i$ is non-elementary, so there exists a
  $g_i\in G_i\setminus E(h_i)$. Let $\underline{g}=(g_1,\dots,g_n)$. 
\cite[Proposition~3.1]{ArzCasTao13} shows that
  $\underline{h}'=\underline{g}\underline{h}^m\underline{\inv{g}}\underline{\inv{h}}^m\in
  N$ is strongly contracting in each coordinate for any sufficiently large
  $m$, so $K$ can be taken to be $\max_i
  d^\pi_{g_i\H_i}(\H_i,g_ih_i^m\inv{g}_i\H_i)+2\xi_i$, where $\xi_i$ is
  chosen by \fullref{lem:projection}. 
\end{proof}

\section{Elements without Long, Positive Projections}\label{sec:nolongprojection}

In the following, let $G$ be any finitely generated, non-elementary group
(not necessarily a product) acting properly discontinuously and cocompactly
by isometries on a based proper geodesic metric space $(\X,d,\bp)$. Suppose
there exists a strongly contracting element $h\in G$ for $G\act\X$. Let
$\H=E(h).\bp$ and let $C$ be the contraction constant for
$\pi_\H$.

Let $D=\diam(G\backslash\X)$ and let $D'=\diam \big(
\left<h\right>\backslash \H \big)$. 

\begin{definition}\label{def:longprojection}
  For $x_0$ and $x_1$ in $\X$, the ordered pair $(x_0,x_1)$ has a
  $K$--\emph{long, positive $h$--projection} if there exists a  $k\in G$ such
  that $d^\pi_{k\H}(x_0,x_1)\geq K$ and $d(k.\bp,\pi_{k\H}(x_0))\leq
  D'$ and there exists
  $\alpha>0$ such that $d(kh^{\alpha}\!.\bp,\pi_{k\H}(x_1))\leq
  D'$.
\end{definition}

It is immediate that the property of having a $K$--long, positive
$h$--projection is invariant under the $G$--action. We also remark that the
`positive' restriction is vacuous if $K>2D'$ and there exists an element of $G$ that
flips the ends of $\H$.

\begin{definition}\label{def:Ghat}
  Let $\hat{G}(K)$ be the elements $g\in G$ such that there exist points
  $x\in B_D(\bp)$ and $y\in B_D(g.\bp)$ and a geodesic $\gamma$ from $x$ to
  $y$ such that no subsegment of $\gamma$ has a $K$--long, positive
  $h$--projection.
\end{definition}

For any $g \in G$, set $|g| = d(o,g.o)$.
\begin{lemma}\label{lem:shorter}
  For all sufficiently large $K$ and for every $g\in
  G\setminus\hat{G}(K)$ there exists a $k\in G$ and an interval $
  [\alpha',\alpha'']\subset\mathbb{Z}^+$ such that
  $|kh^{-\alpha}\inv{k}g|<|g|$ for all $\alpha'\leq \alpha\leq
  \alpha''$.
The lower bound $\alpha'$ depends only on $h$ and the upper bound
$\alpha''$ depends linearly on $K$. 
\end{lemma}
\begin{proof}
    Let $\gamma$ be a geodesic from $\gamma(0)=\bp$ to $\gamma(T)=g.\bp$.
  Since $g\notin \hat{G}(K)$, there are times $t_0$ and $t_1$ in $[0,T]$
  such that $(\gamma(t_0),\gamma(t_1))$ has a $K$--long, positive
  $h$--projection. Let $k\in G$ such that
  $d^\pi_{k\H}(\gamma(t_0),\gamma(t_1))\geq K$ and
  $d(\pi_{k\H}(\gamma(t_0)), k.\bp)$, and let $\beta>0$ be such that
  $d(kh^\beta,\pi_{k\H}(\gamma(t_1)))\leq D'$.

Let $\lambda$, $\epsilon$, and $\mu$ be the constants of
\fullref{lem:almostgeodesic} for $h$.
Let $\xi$ be the constant of \fullref{lem:projection}
Since $i\mapsto h^i\!.\bp$ is $(\lambda,\epsilon)$--quasi-geodesic and
$d(1,h^\beta\!.\bp)\geq K-2D'-2\xi$ we have $\beta\geq \frac{1}{\lambda}(K-2D'-2\xi)-\epsilon$.

Set $\alpha''=\beta$ and
$\alpha'=\lambda(4(C+D'+\xi)+\epsilon+2\mu+1))$.
We assume that $K$ is large enough so that $\alpha''\geq\alpha'$.
 For all
$\alpha'\leq\alpha\leq\alpha''$ we have:
\[d(k.\bp,kh^{\beta}\!.\bp)\geq d(k.\bp,kh^{\alpha}\!.\bp)+d(kh^{\alpha}\!.\bp,kh^{\beta}\!.\bp)-2\mu\]
Rearranging, and using the quasi-geodesic condition for $k\H$:
\begin{align*}
  d(kh^{\alpha}\!.\bp,kh^{\beta}\!.\bp)&\leq
  d(k.\bp,kh^{\beta}\!.\bp)-d(k.\bp,kh^{\alpha}\!.\bp)+2\mu\\
&\leq d(k.\bp,kh^{\beta}\!.\bp)-(\sfrac{\alpha}{\lambda}-\epsilon)+2\mu\\
&<d(k.\bp,kh^{\beta}\!.\bp) -4(C+D'+\xi)
\end{align*}
Now we use the fact that $\gamma$ passes $C+D'+\xi$ close to $k.\bp$ and $kh^\beta\!.\bp$:
  \begin{align*}
    |g|
    &= d
    \big(\gamma(0),\gamma(t_0))
     + d \big( \gamma(t_0),\gamma(t_1) \big) + d\big( \gamma(t_1),\gamma(T) \big)\\
    &\geq d \big( \gamma(0),\gamma(t_0) \big) 
     +  d \big( \gamma(t_0),k.\bp \big)+d(k.\bp,kh^{\beta}\!.\bp) \\
 &\qquad\qquad+ d \big( kh^{\beta}\!.\bp,\gamma(t_1)
    \big) + d \big( \gamma(t_1),\gamma(T) \big)- 4(C+D'+\xi)\\
  \end{align*}
So:
  \begin{align*}
    |kh^{-\alpha}\inv{k}g|&\leq
    d(\gamma(0),\gamma(t_0))+d(\gamma(t_0),k.\bp)+d(k.\bp,kh^{-\alpha}\inv{k}kh^{\beta}\!.\bp)\\
    &\qquad\qquad+d(kh^{-\alpha}\inv{k}kh^{\beta}\!.\bp,kh^{-\alpha}\inv{k}\gamma(t_1))+d(kh^{-\alpha}\inv{k}\gamma(t_1),kh^{-\alpha}\inv{k}\gamma(T))\\
    &=d(\gamma(0),\gamma(t_0))+d(\gamma(t_0),k.\bp) +d(kh^{\alpha}\!.\bp,kh^{\beta}\!.\bp)\\
    &\qquad\qquad+d(kh^{\beta}\!.\bp,\gamma(t_1))+d(\gamma(t_1),\gamma(T))\\ 
    &\leq |g| +4(C+D'+\xi)-d(k.\bp,kh^{\beta}\!.\bp) +d(kh^{\alpha}\!.\bp,kh^{\beta}\!.\bp)\\
    &<|g|\qedhere
  \end{align*}
\end{proof}

\begin{lemma}\label{lem:subadditivity}
  Fix $K$ and let $P(r) = \#(B_r(\bp)\cap\hat{G}(K).\bp)$. The function
  $\log P(r)$ is subadditive in $r$, up to bounded error.
\end{lemma}

\begin{proof}
  Let $g.\bp\in B_r(\bp)\cap\hat{G}(K).\bp$. Let $x$, $y$, and $\gamma$ be
  as in \fullref{def:Ghat}. Let $m+n=r$. 
If $d(x,y)>m$ let $z$ be the
  point on $\gamma$ at distance $m$ from $x$. 
Otherwise, let $z=y$. There
  exists an $f\in G$ such that $d(z,f.\bp)\leq D$.

  We claim that $f$ contributes to $P(m+2D)$ and $\inv{f}g$ contributes to
  $P(n+2D)$. 
This is because $d(\bp,f.\bp)\leq m+2D$, and the subsegment
  of $\gamma$ from $x$ to $z$ is a geodesic for $f$ satisfying
  \fullref{def:Ghat}. 
Similarly, $d(\bp,\inv{f}g.\bp)=d(f.\bp,g.\bp)\leq
  n+2D$, and the subsegment of $\inv{f}.\gamma$ from $\inv{f}.z$ to $\inv{f}.y$ is a geodesic for
  $\inv{f}g$ satisfying \fullref{def:Ghat}.

  This shows that for any $m+n=r$ we have $P(r)\leq P(m+2D)\cdot P(n+2D)$.
  Applying this relation for $(m-2D)+4D=m+2D$ and $(n-2D)+4D=n+2D$ yields:
  \[P(r)\leq (P(6D))^2\cdot P(m)\cdot P(n)\]
  Thus:
  \[\log P(m+n)\leq \log P(m)+\log P(n) +2\log P(6D). \qedhere\] 
\end{proof}

There is a result known as Fekete's Lemma that says if $(a_i)$ is a
subadditive sequence then $\lim_{i\to\infty}\frac{a_i}{i}$ exists and
is equal to $\inf_i\frac{a_i}{i}$.
We will need the following generalization for almost subadditive sequences:
\begin{lemma}\label{lem:Fekete}
  Let $(a_i)$ be an unbounded, increasing sequence of positive numbers. Suppose
  there exists $b$ such that $a_{m+n}\leq a_m+a_n+b$ for all $m$ and
  $n$.
Then $L=\lim_{i\to\infty}\frac{a_i}{i}$ exists and $a_i\geq Li-b$ for
all $i$.
\end{lemma}
\begin{proof}
  Let $L^+=\limsup_i\frac{a_i}{i}$. Let $L^-=\liminf_i\frac{a_i}{i}$.
Suppose that $L^+>L^-$\!. Let $\epsilon=\frac{L^+-L^-}{3}$.
Since the sequence is increasing and unbounded, there exists an $I$
such that for all $i>I$ we have
$\frac{a_i+b}{a_i}<\sqrt{\frac{L^+-\epsilon}{L^-+\epsilon}}$.
Fix an $i>I$ such that $\frac{a_i}{i}<L^-+\epsilon$.
Choose a $j$ such that $\frac{a_j}{j}>L^+-\epsilon$ and
$\frac{q+1}{q}<\sqrt{\frac{L^+-\epsilon}{L^-+\epsilon}}$, where
$qi\leq j<(q+1)i$.
\[L^+-\epsilon<\frac{a_j}{j}\leq\frac{a_j}{qi}<\frac{(q+1)(a_i+b)}{qi}<\frac{L^+-\epsilon}{L^-+\epsilon}\cdot\frac{a_i}{i}\leq \frac{L^+-\epsilon}{L^-+\epsilon}(L^-+\epsilon)=L^+-\epsilon\]
This is a contradiction, so $L=L^+=L^-$\!.

If for some $i$ we have $a_i<Li-b$ then
\[L=\lim_{j\to\infty}\frac{a_{ij}}{ij}\leq\lim_{j\to\infty}\frac{j(a_i+b)}{ij}=\frac{a_i+b}{i}<L,\]
which is a contradiction.
\end{proof}

\subsection{Divergence}
For any subset $A \subset G$, define: 
\[ \Theta_A(s) = \sum_{r=0}^\infty
\# \big( B_r(o) \cap A.\bp \big) e^{-rs}\] 
The growth exponent $\delta_A$
is the \emph{critical exponent} of $\Theta_A$, that is, $\Theta_A$ diverges for all $s <
\delta_A$ and converges for all $s > \delta_A$. 
We say $A$ is
\emph{divergent} if $\Theta_A$ diverges at $\delta_A$.

\begin{lemma}\label{lemma:divergence}
  $\hat{G}(K)$ is divergent.
\end{lemma}

\begin{proof}
Let $P(r)=\#(B_r(\bp)\cap\hat{G}(K).\bp)$.
  By \fullref{lem:subadditivity} and \fullref{lem:Fekete}, $\log P(r)
  \geq r\delta_{\hat{G}(K)}-2\log P(6D)$ for all $r$. Thus:
  \[\Theta_{\hat{G}(K)}(\delta_{\hat{G}(K)})=\sum_{r=0}^\infty
  P(r)\exp(-r\delta_{\hat{G}(K)})\geq
  \sum_{r=0}^\infty \frac{1}{P(6D)^2}=\infty. \qedhere\] 
\end{proof}

\begin{lemma}\label{mainlemma}
  For sufficiently large $K$, the growth exponent of $\hat{G}(K)$ is
  strictly smaller than the growth exponent of $G$. 
\end{lemma}

\begin{proof}
  Let $h'\in G$ and $D$ be the element and constant, respectively, of
  \fullref{lem:independentcontractingelements} (in this case the product
  has only one factor). Let $K>D$.

  Define a map $\phi$ on $\hat{G}(K)$ as follows.
  \[\phi(g)=\begin{cases}
  h'g\inv{h}' &\text{ if }d^\pi_{\H}(\bp,g.\bp)\geq K\text{ and 
  }d^\pi_{g\H}(\bp, g.\bp)\geq K\\
  h'g &\text{ if }d^\pi_{\H}(\bp,g.\bp)\geq K\\
  g\inv{h}' &\text{ if }d^\pi_{g\H}(\bp, g.\bp)\geq K\\
  g &\text{ otherwise}
  \end{cases}\]

  Then for all $g\in\hat{G}(K)$ we have $d^\pi_{\H} \big( \bp,\phi(g).\bp
  \big)< K$ and $d^\pi_{\phi(g)\H} \big( \bp, \phi(g).\bp \big)< K$.

  Let $\hat{G}'(K)$ be the image of $\phi$. Then $\phi$ is a bijection between
  $\hat{G}(K)$ and $\hat{G}'(K)$, and for all $g\in\hat{G}(K)$ we have
  $|g|=|\phi(g)|\pm 2|h'|$. It follows that
  $\delta_{\hat{G}(K)}=\delta_{\hat{G}'(K)}$ and $\hat{G}'(K)$ is divergent.

  Let $Z$ be a maximal $2K$--separated subset of $\hat{G}'(K)$. Then
  $\delta_Z=\delta_{\hat{G}'(K)}$ and $Z$ is divergent. For $z$ and $z'$ in
  $Z$, if $z\H=z'\H$ then since $d^\pi_{z\H}(\bp,z.\bp)<K$ and
  $d^\pi_{z'\H}(\bp,z'\!.\bp)<K$ we have $d(z.\bp,z'\!.\bp)<2K$, so $z=z'$.

  Consider the free product set 
  \[ Z^*\!*h^m = \bigcup_{i=1}^\infty \big\{ z_1h^m\cdots z_ih^m \mid z_j\in
  Z\setminus \{1\} \big\}. \] 
  By the same arguments as \cite[Proposition~4.1]{ArzCasTao13}, for all
  sufficiently large $m$, the orbit map is an injection of $Z^*\!*h^m$ into
  $\X$. 
This fact, together with divergence of $Z$, implies that
  $\delta_Z<\delta_G$, by \cite[Criterion~2.4]{DalPeiPic11}.
\end{proof}

\section{Proof of the Main Theorem}\label{sec:proof}

Let $(\X_1,d_1,\bp_1), \ldots, (\X_n,d_n,\bp_n)$ a finite collection of proper geodesic
metric spaces. 
Let $\X = \X_1 \times \cdots \times \X_n$, and let $\bp=(\bp_1,\ldots,\bp_n)$.
Let $\underline{x}=(x_1,\ldots,x_n)$ and $\underline{y}=(y_1,\ldots,y_n)$ be any points in
$\X$. 
For any $1\leq p <\infty$, the $L^p$ metric on $\X$ is defined to be: 
\[ d^p(\underline{x},\underline{y}) = \left( \sum_{i=1}^n (d_i(x_i,y_i))^p \right)^{1/p}\] 
The  $L^\infty$ metric on $\X$ is:
\[ d^\infty(\underline{x},\underline{y}) = \max_{i} d_i(x_i,y_i)\]

\begin{proposition}\label{prop:duality}
For $i=1,\dots, n$, let $G_i$ be a non-elementary, finitely generated group acting properly
discontinuously and cocompactly by isometries on a proper geodesic
metric space $\X_i$. 
Let $G=G_1\times\cdots\times G_n$.
For each $i$, let $A_i$ be a subset of $G_i$ such that $\log P_i(r)$
is subadditive in $r$, up to bounded error, for $P_i(r)=
\#(B_r(\bp_i)\cap A_i.\bp_i)$. 
 Let $\delta_i=\delta_{A_i.\bp_i}$ be the
  growth exponent of $A_i$.
For $1\leq p\leq\infty$, the growth exponent $\delta_A$ of $A=\prod_i^n A_i$ with respect
to the $L^p$ metric on $\X$ is the
$L^q$--norm of $(\delta_1,\dots, \delta_n)$, where
$\sfrac{1}{p}+\sfrac{1}{q}=1$, and $\frac{1}{\infty}$ is understood to
be $0$.
\end{proposition}

\begin{proof}
For each $g \in G_i$ let $|g|_i =
d_i(\bp_i,g.\bp_i)$. 
For $\underline{g}=(g_1,\ldots,g_n) \in G$, let
$|\underline{g}|_p = d^p(\bp,\underline{g}.\bp)$.
Let $B_r^p$ be the closed $r$--ball with respect to the $L^p$ metric.

  Let $P(r)=\# B_r^p(\bp)\cap A.\bp$.

Let $\mathbb{R}^n$ be equipped with
  the $L^p$ norm $||\cdot||_p$, and let  $S_r^p$ be the vectors of norm $r$. 
Let $\phi \from \mathbb{R}^n \to
  \mathbb{R}$ be the linear function $\phi(x_1,\ldots,x_n) = \sum_{i=1}^n
  \delta_i x_i$. For every $r>0$ the duality of $L^q$ and $L^p$ implies: 
  \[ ||(\delta_1,\dots,\delta_n)||_q = ||\phi||_p =
    \sup_{(x_1,\ldots, x_n) \in S^p_r}\frac{|\phi(x_1,\ldots, x_n)|}{r}\]

Since $\delta_i \geq
  0$ for all $i$, the supremum can be restricted to the positive sector of
  $S_r^p$. 
Furthermore, letting 
\[ Z_r^p = \left\{ (r_1,\ldots, r_n) \mid ||(r_1,\dots,r_n)||_p\leq r,\, r_i \in \mathbb{N} \right\}, \] 
we have: 
  \[ ||\phi||_p = \lim_{r \to \infty} \max_{(r_1,\ldots,r_n) \in Z_r^p}
  \frac{\phi(r_1,\ldots,r_n)}{r}\]   
  
Given two positive valued functions $f(r)$ and $g(r)$, we write 
  $f(r) \sim g(r)$ if $\lim_{r\to \infty} \frac{\log f(r)}{\log g(r)} = 1$.

\fullref{lem:subadditivity} and \fullref{lem:Fekete} imply $P_i(r)
\sim e^{\delta_i r}$ for each
  $i=1,\ldots,n$, so: 
  \begin{align*}
  ||\phi||_p&=\lim_{r \to \infty} \max_{(r_1,\ldots,r_n) \in Z_r^p}
  \frac{\phi(r_1,\ldots,r_n)}{r}\\
&=\lim_{r\to\infty}\max_{(r_1,\ldots,r_n) \in Z_r^p}  \frac{\log \prod_{i=1}^n P_i(r_i)}{r}
  \end{align*}

  For any fixed $r$ there is $(z_{r,1},\ldots, z_{r,n}) \in Z_r^p$  such that:
  \[ \prod_{i=1}^n P_i(z_{r,i}) = \max_{(r_1,\ldots,r_n) \in Z_r^p}  \prod_{i=1}^n P_i(r_i) \]

We also note that: 
\[ \prod_{i=1}^n
  P_i(z_{r,i}) \leq P(r) \le \sum_{(r_1,\ldots,r_n)\in Z_r^p} \prod_{i=1}^n
    P_i(r_i) \le
  \#Z_r^p \cdot\prod_{i=1}^n P_i(z_{r,i}) \] 

Since $\# Z_r^p\leq r^n$, this means $P(r) \sim \prod_i^n P_i(z_{r,i})$. 

Therefore:
  \begin{align*}
    \delta_A=\limsup_{r\to \infty} \frac{\log P(r)}{r} 
    & = \lim_{r \to \infty} \frac{\log \prod_{i=1}^n
      P_i(z_{r,i})}{r} \\
    &=\lim_{r\to\infty}\max_{(r_1,\ldots,r_n) \in Z_r^p}  \frac{\log \prod_{i=1}^n P_i(r_i)}{r}\\
    &=||\phi||_p= ||(\delta_1,\dots,\delta_n)||_q\qedhere
  \end{align*}
\end{proof}

\begin{proof}[Proof of \fullref{main}]
The existence of a strongly contracting element implies that each
factor group has strictly positive growth exponent, and
the main theorem of
\cite{ArzCasTao13} says that $G_i\act\X_i$ is growth tight, so we are done
if $n=1$.

Assume $n>1$ and let $1\leq q\leq\infty$ be such that $\sfrac{1}{p}+\sfrac{1}{q}=1$.
If $p=1$, then
by \fullref{prop:duality} the growth exponent of $G$
is the maximum of the growth exponents of the $G_i$.
Thus, we may kill the slowest growing factor without changing the
growth exponent, and the action of $G$ on $\X$ with the $L^1$ metric is
not growth tight.

Now assume $p>1$. 
Let $\chi_i\from G\to G_i$ be projection to the $i$--th coordinate. 
Let $N$
be an infinite normal subgroup of $G$. 

First we assume that $\chi_i(N)$ is infinite for all $i$.

By \fullref{lemma:contracting}, there exists an element
$\underline{h}=(h_1,\dots, h_n)\in N$ such that $h_i$ is a strongly
contracting element for $G_i\act X_i$ for each $i$.

Let $A$ be a \emph{minimal section} of the quotient map $G\to G/N$. That is,
$A$ consists of a representative for each coset $gN$ and
$d(\bp,\underline{a}.\bp) = d(N.\bp,\underline{a}N.\bp)$
for all $\underline{a} \in A$, where $d$ is the $L^p$ metric on $\X$. 

\begin{proposition} \label{prop:minimal}
  
  For all sufficiently large $K$ and
    for all $\underline{a}=(a_1,\ldots,a_n) \in A$ there exists an
    index $1\leq i\leq n$ such that $a_i \in \hat{G}_i(K)$.  

\end{proposition}

\begin{proof}

  For each $i$, let $\hat{G_i}(K)$ be as in \fullref{def:Ghat} for
  each $G_i$. 
Assume $K$ is greater than the constants $K$ from \fullref{mainlemma}
and \fullref{lem:shorter} applied to each $G_i$.

  Suppose $\underline{a}$ is such that for all $i$ we have $a_i\in G_i\setminus \hat{G_i}(K)$.
  For each $i$, let $k_i\in G_i$ and $[\alpha_i',\alpha_i'']$ be the
  $k$ and interval, respectively,
  from \fullref{lem:shorter} applied to $a_i$. 
The $\alpha_i'$ depend only on their respective $h_i$, while the
$\alpha_i''$ depend linearly on $K$.
By choosing $K$ large enough, we may choose $\alpha$ such that $\max_i
\alpha_i'\leq \alpha\leq\min_i\alpha_i''$, so that
$\alpha\in[\alpha_i',\alpha_i'']$ for all $i$.
Let $\underline{k}=(k_1,\dots,k_n)$.
The $i$--th coordinate of
$\underline{k}\underline{\inv{h}}^{\alpha}\inv{\underline{k}}\underline{a}$
is $k_i\inv{h_i}^{\alpha}\inv{k_i}a_i$, which is shorter than $a_i$ by
\fullref{lem:shorter}.
But this means that
$\underline{k}\underline{\inv{h}}^{\alpha}\inv{\underline{k}}\underline{a}$
is shorter than $\underline{a}$.
This contradicts the fact that $\underline{a}$ belongs to a minimal
section, since $\underline{k}\underline{\inv{h}}^{\alpha}\inv{\underline{k}}\underline{a}=\underline{a}(\underline{\inv{a}}\underline{k}\underline{\inv{h}}^{\alpha}\inv{\underline{k}}\underline{a})\in\underline{a}N$.
\end{proof}

Continuing the proof of \fullref{main}, by \fullref{prop:minimal}, \[A\subset
\bigcup_{i=1}^nG_1\times\cdots\times\hat{G}_i\times\cdots\times G_n,\]
where $\hat{G}_i=\hat{G}_i(K)$ for some sufficiently large $K$.
By \fullref{prop:duality}, the growth exponent of
$G_1\times\cdots\times\hat{G}_i\times\cdots\times G_n$ is
$||(\delta_1,\dots,\hat{\delta_i},\dots,\delta_n)||_q$, where
$\delta_i$ is the growth exponent of $G_i$ and $\hat{\delta_i}$ is the
growth exponent of $\hat{G}_i$.
Thus, the growth exponent of $A$ is $\max_i ||(\delta_1,\dots,\hat{\delta_i},\dots,\delta_n)||_q$.
By \fullref{mainlemma}, $\hat{\delta_i}< \delta_i$ for each $i$,
  so, since $q<\infty$: \[\delta_{G/N}=\delta_A=\max_i ||(\delta_1,\dots,\hat{\delta_i},\dots,\delta_n)||_q < ||(\delta_1,\dots,\delta_n)||_q  = \delta_G\]

  It remains to consider the case that some $\chi_i(N)$ is
  finite. 
By reordering, if necessary, we may assume $\chi_i(N)$ is finite for
$i\leq m$ and infinite for $i>m$.
Since $N$ is infinite, $m<n$.
Let $G^1=G_1\times\cdots\times G_m$ with
$\chi^1=\chi_1\times\cdots\times \chi_m\from G\to G^1$.
Let $G^\infty=G_{m+1}\times\cdots\times G_n$ with
$\chi^\infty=\chi_{m+1}\times\cdots\times\chi_n\from G\to G^\infty$.

Now $\ker(\chi^1) \cap N$ is a finite index subgroup of
  $N$ that is normal in $G$, so $G/N$ is a quotient of $G/(\ker(\chi^1) \cap
  N)$ by a finite group, and they have the same growth rates. 
Replacing
  $N$ with $\ker(\chi^1) \cap N$, we can assume that $\chi_i(N)$
  is trivial for $1\leq i\leq m$ and infinite for $m<i\leq n$. 
The theorem applied to $G^\infty$ shows that
$\delta_{G^\infty/\chi^\infty(N)}<\delta_{G^\infty}$, so, since $q<\infty$:
\[\delta_{G/N}=||(\delta_{G^1},\delta_{G^\infty/\chi^\infty(N)})||_q<||(\delta_{G^1},\delta_{G^\infty})||_q=\delta_G\qedhere\]
\end{proof}

In the case that the normal subgroup has infinite projection to each
factor, our proof uses the existence of a contracting element in each
factor in an essential way. 
One wonders if the theorem is still true without this hypothesis:
\begin{question*}
  If, for $1\leq i\leq n$, $G_i$ is a non-elementary, finitely generated group acting
  properly discontinuously and cocompactly by isometries on a proper
  geodesic metric space $\X_i$, and if, for all $i$, $G_i\act\X_i$ is
  growth tight, is it still true that the product group is growth
  tight with respect to the action on the
  product space with the $L^p$ metric for some/all $p>1$?
\end{question*}



\bibliographystyle{hyperamsplain}
\bibliography{ProductGrowthTight}

\end{document}